\documentclass[11pt]{amsart}

\usepackage{amsmath}
\usepackage{amstext}
\usepackage{amssymb}
\usepackage{amsthm}
\usepackage{enumerate}
\usepackage{bbm}
\usepackage{comment}
\usepackage[USenglish]{babel}
\usepackage[utf8]{inputenc}
\usepackage[T1]{fontenc}

\makeatletter
\def\imod#1{\allowbreak\mkern10mu({\operator@font mod}\,\,#1)}
\makeatother

\theoremstyle{plain}
\newtheorem{thm}{Theorem}
\theoremstyle{definition}
\newtheorem{defi}[thm]{Definition}
\newtheorem{cor}[thm]{Corollary}
\newtheorem*{cor*}{Corollary}
\newtheorem*{lem*}{Lemma}
\newtheorem*{thm*}{Theorem}
\newtheorem*{clm*}{Claim}
\newtheorem{lem}[thm]{Lemma}

\newtheorem{clm}[thm]{Claim}

\newtheorem{rem}[thm]{Remark}

\DeclareMathOperator{\id}{id}
\DeclareMathOperator{\fil}{fil}

\DeclareMathOperator{\FF}{FF}

\renewcommand{\id}{\mathrm{id}}

\newcommand{\bbP}{\mathbb{P}}
\newcommand{\bbQ}{\mathbb{Q}}

\renewcommand{\phi}{\varphi}

\newcommand{\calA}{\mathcal{A}}
\newcommand{\calB}{\mathcal{B}}
\newcommand{\calC}{\mathcal{C}}

\newcommand{\calE}{\ensuremath{\mathcal{E}}}
\newcommand{\calF}{\mathcal{F}}

\newcommand{\calH}{\ensuremath{\mathcal{H}}}

\newcommand{\ra}{\rangle}

\newcommand{\concatB}{\mathbin{\rotatebox[origin=c]{90}{\scalebox{.7}{(\kern1ex)}}}}

\date{\today}
\begin{document}

\title[Selective independence]{Selective independence} 

\author{Vera Fischer}

\address{Institute of Mathematics, University of Vienna, Augasse 2-6, 1090 Vienna, Austria}
\email{vera.fischer@univie.ac.at}

\thanks{\emph{Acknowledgements:} The author would like to thank the Austrian Science Fund (FWF) for the generous support through grant number Y1012-N35.}





\begin{abstract} 
Let $\mathfrak{i}$ denote the minimal cardinality of a maximal independent family and let $\mathfrak{a}_T$ denote the minimal cardinality of a maximal family of pairwise almost disjoint subtrees of $2^{<\omega}$. Using a countable support iteration of proper, $^\omega\omega$-bounding posets of length $\omega_2$ over a model of CH, we show that consistently $\mathfrak{i}<\mathfrak{a}_T$. Moreover, 
we show that the inequality can be witnessed by a co-analytic maximal independent family of size $\aleph_1$ in the presence of a $\Delta^1_3$ definable well-order of the reals. The main result of the paper can be viewed as a partial answer towards the well-known open problem of the consistency of $\mathfrak{i}<\mathfrak{a}$.
\end{abstract}

\maketitle

\section{Introduction}

In this paper, we will study two infinitary combinatorial structures of the real line, independent and almost disjoint families. A family $\mathcal{A}$ of $[\omega]^\omega$ is said to be independent if for every two disjoint finite non-empty subfamilies $\mathcal{B}$ and $\mathcal{C}$ of $\mathcal{A}$ the set $\bigcap\calB\backslash\bigcup\calC$ is infinite. It is maximal independent if it is in addition maximal under inclusion. The minimal cardinality of a maximal independent family is denoted $\mathfrak{i}$. Another, well-known cardinal characteristic of the real line is the almost disjointness number, denoted $\mathfrak{a}$ and defined as the minimal cardinality of a maximal (under inclusion) infinite family of pairwise disjoint infinite subsets of $\omega$. 

The consistency  of $\mathfrak{a}<\mathfrak{i}$ is well-known, as it holds in the Cohen model.  The consistency of $\mathfrak{i}<\mathfrak{a}$ is however a long-standing open problem. Since $\mathfrak{d}\leq\mathfrak{i}$ (see~\cite{LH}), a model of $\mathfrak{i}<\mathfrak{a}$ is necessarily a model of $\mathfrak{d}<\mathfrak{a}$. However, in all known models of $\mathfrak{d}<\mathfrak{a}$ (using ultrapowers, or Shelah's template construction,~\cite{SS0}) the value of $\mathfrak{i}$ coincides with the value of $\mathfrak{a}$. For more recent studies on the set of possible 
cardinalities of maximal independent families, see~\cite{VFSS1, VFSS2}.

Building upon earlier work ~\cite{VFDM1, JBVFYK}, we define a class of independent families (originally appearing in~\cite{SS1}), to which we refer as selective independent families (see Definition~\ref{selective.independence.def}). Selective independent families exist under CH (shown originally in~\cite{SS1} and later~\cite{VFDM1}). Our results show that their existence is also consistent with the negation of CH. Moreover, we show that selective independent families are preserved under the countable support iteration of perfect trees posets (see Definition~\ref{perfect.tree.poset} and Theorem~\ref{preservation.thm}). 

A not that well-known relative of the almost disjointness number is the number $\mathfrak{a}_T$, which is defined as the minimal cardinality of a maximal family of pairwise disjoint subtrees of $2^{<\omega}$. Relying on our preservation theorem, we establish the following:

\begin{thm*} Assume CH. It is relatively consistent that $\mathfrak{i}<\mathfrak{a}_T$.
\end{thm*}

The above generic extension is a model of $\mathfrak{d}=\omega_1<\mathfrak{a}_T=\omega_2$, result which was first obtained by O. Spinas in~\cite{OS}.

Elaborating on Miller's technique from~\cite{AMiller} and introducing the notion of an indestructibility tower, in~\cite{JBVFYK} the authors construct a co-analytic Sacks indestructible maximal independent family. The family is in fact selective independent and thus subject to the more general preservation theorem developed in this paper. In~\cite{VFSF} the method of coding with perfect trees (and localization) is used to show that certain cardinal invariant constellations are consistent with the existence of a $\Delta^1_3$ definable wellorder of the reals. Both coding with perfect trees and localization fall in the framework of our general preservation theorem, leading to the following result:

\begin{thm*} The inequality $\mathfrak{i}<\mathfrak{a}_T$ is consistent with the existence of a $\Delta^1_3$-wellorder of the reals. Additionally, the independence number can be witnessed by a co-analytic maximal independent family.
\end{thm*}

Note that both, the $\Delta^1_3$-definable wellorder of the reals, as well as the $\Pi^1_1$ definable maximal independent family, are maximal in their projective complexity. Indeed, Miller~\cite{AMiller} has shown that there are no analytic maximal independent families, while the existence of a $\Sigma^1_2$ definable wellorder of the reals implies by a theorem of Mansfield that every real is constructible (see also~\cite{VFSF}).

{\emph{Outline of the paper:}} In Section 2 we study equivalent formulations of dense maximality and the density ideal. We show that in iterations of perfect set posets every element in the density ideal contains a ground model set from the density ideal. In Section 3 we introduce the notion of selective independent families and establish our main preservation theorem. Sections 4 considers applications to compact partitions of the Baire space and the number $\mathfrak{a}_T$. Section 5 contains applications to models with projective wellorders of the reals. 
Section 6 contains concluding remarks and open questions.

\section{Dense Maximality}

\begin{defi} Let $\calA$ be an infinite independent family. Then:
	\begin{enumerate}
		\item The density ideal of $\calA$, denoted $\id(\calA)$ consists of all $X\subseteq \omega$ with the property that $\forall h\in\FF(\calA)$ there is  $h'\in \FF(\calA)$ such that $h'\supseteq h$ and $\calA^{h'}\cap X=\emptyset$.
		\item The density filter of $\calA$, denoted $\fil(\calA)$, is the dual filter of $\id(\calA)$. Thus $Y\in \fil(\calA)$ if and only if $\forall h\in\FF(\calA)\exists h'\in\FF(\calA)$ such that $h'\supseteq h$ and $\calA^{h'}\subseteq Y$.	
	\end{enumerate}
\end{defi}

\begin{lem}\label{dense.max.equiv} Let $\calA$ be an infinite independent family. The following are equivalent:
	\begin{enumerate}
		\item For all $X\in\mathcal{P}(\omega)\backslash \calA$ and all $h\in\FF(\calA)$ there is an extension $h'$ of  $h$ such that $\calA^{h'}\cap X$ or $\calA^{h'}\backslash X$ is finite (and so empty).
		\item For all $h\in\FF(\calA)$ and all $X\subseteq \calA^h$ either there is $B\in\id(\calA)$ such that $\calA^h\backslash X\subseteq B$, or there is $h'\in\FF(\calA)$ such that $h'\supseteq h$ and $\calA^{h'}\subseteq \calA^h\backslash X$.
		\item For each $X\in\mathcal{P}(\omega)\backslash\fil(\calA)$ there is $h\in\FF(\calA)$ such that $X\subseteq \omega\backslash \calA^h$.  
   \end{enumerate}		
\end{lem}
\begin{proof}
The equivalence of $(1)$ and $(2)$ can be found in~\cite[Lemma 31]{VFDM2}. The equivalence of $(2)$ and $(3)$ implicitly appears in~\cite[Theorem 29]{VFDM1}, as well as~\cite{SS1}
\end{proof}

\begin{defi} An independent family $\calA$ is said to be densely maximal if any of the above three equivalent characterisations hold.
\end{defi}

We agree to use the following terminology: 
	
\begin{rem}\label{perfect.tree.poset} A partial order is said to be a perfect tree poset if its elements are perfect subtrees of $2^{<\omega}$ and satisfies Axiom A.
\end{rem}

Thus perfect tree posets are proper, $^\omega\omega$-bounding and have the Sacks property. The density ideal $\id(\calA)$ associated to an arbitrary (infinite) independent family has the following property:

\begin{lem}\label{single.step.density} Assume CH. Let $\calA$ be an independent family, let $\bbQ$ be a perfect tree poset and let $H$ be a $\bbQ$-generic filter. Then
$\id(\calA)^{V[H]}$ is generated by $\id(\calA)\cap V$, where $V$ denotes the ground model. With other words for each $X\in\id(\calA)\cap V[H]$ there is $Y\in\id(\calA)\cap V$ such that $X\subseteq Y$. 
\end{lem}
\begin{proof} Fix $p\in\bbQ$ such that $p\Vdash\dot{X}\in\id(\calA)$. Using a fusion sequence find $q\leq p$ and a countable $\calB\subseteq\calA$ such that $q\Vdash\dot{X}\in\id(\calB)$. Identifying $\FF(\calB)$ with $2^{<\omega}$ we get a name $\dot{D}$ for a dense open subset of $2^{<\omega}$ defined by $q\Vdash h\in \dot{D}\hbox{ iff } \calB^h\cap \dot{X}=\emptyset$.

\begin{clm*} In $V$ the set $\Delta$ of $r\in\bbQ$ such that $r\Vdash \check{K}\subseteq\dot{D}$ for some dense $K\subseteq 2^{<\omega}$ is dense below $q$. 
\end{clm*}
\begin{proof} Let $\langle s_n:n\in\omega\ra$ be an enumeration of $2^{<\omega}$. Inductively construct a fusion sequence $\langle S_n:n\in\omega \ra$ in $\bbQ$ and a sequence $\langle t_n:n\in\omega\ra$ of elements of $2^{<\omega}$ such that $t_n\supseteq s_n$ for each $n$. 
Let $S_0=q$. Suppose we have constructed $S_n$ and $t_n\supseteq s_n$ such that for every extension $R$ of $S_n$ we have $R\Vdash_{\bbQ} t_n\in\dot{D}$. Let $\langle u_j:j\in 2^n\ra$ enumerate $\hbox{split}_n(S_n)$. Consider $(S_n)_{u_0}$. Then $(S_n)_{u_0}\Vdash"\dot{D}\hbox{ is dense}$ and so there is $U_0\leq (S_n)_{u_0}$ and $t'_0\in 2^{<\omega}$ such that $t'_0\supseteq s_{n+1}$ and $U_0\Vdash t'_0\in\dot{D}$. Next, consider $(S_n)_{u_1}$ and find $t'_1\supseteq t'_0$, 
$U_1\leq (S_n)_{u_1}$ such that $U_1\Vdash t'_1\in\dot{D}$. Proceed inductively and find $t'_{2^n-1}$ extending $t'_{2^n-2}$, and a condition $U_{2^n-1}\leq (S_n)_{u_{2^n-1}}$ which forces that $t'_{2^n-1}\in\dot{D}$. Then take $t_{n+1}=t'_{2^n-1}$, 
$S_{n+1}=\bigcup\{U_j:j\in 2^n\}$ and note that for every extension $R$ of $S_{n+1}$ we have that $R\Vdash t_{n+1}\in\dot{D}$.  Finally, let $r$ be the fusion of the sequence $\langle S_n:n \in\omega\ra$ and let $K=\{t_n\}_{n\in\omega}$. Then $K$ is dense in $2^{<\omega}$ and $r\Vdash \check{K}\subseteq \dot{D}$.
\end{proof}

Then for some dense  $K\subseteq 2^{<\omega}$ we have $V[H]\vDash K\subseteq\dot{D}[H]$. Take $Y=\bigcap_{t\in K}(\omega\backslash \calA^t)$. Then $Y\in\id(\calA)\cap V$ and $V[H]\vDash \dot{X}[H]\subseteq Y$ as desired.
\end{proof}

Moreover the above property generalizes to the countable support iterations of perfect trees posets (see also~\cite{SS1}).	
	
\begin{lem}\label{ideal.generators} Assume CH. Let $\calA$ be an independent family and let $\langle \bbP_\alpha,\dot{\bbQ}_\beta:\alpha\leq\omega_2,\beta<\omega_2\ra$ be a countable support iteration of perfect trees posets. Then for every $\alpha\leq\omega_2$ the ideal $\id(\calA)^{V^\bbP_\alpha}$ is generated by $\id(\calA)\cap V$. That is, $V^{\bbP_\alpha}\vDash(\forall X\in\id(\calA)\exists Y\in\id(\calA)\cap V\hbox{ such that } X\subseteq Y)$.
\end{lem}
\begin{proof}
The proof is a straightforward corollary to the following more general statement:

\begin{clm*} Let $\langle \bbP_\alpha,\dot{\bbQ}_\beta: \alpha\leq\omega_2,\beta<\omega_2\ra$ be a countable support iteration of perfect trees posets. Then every dense subset of $2^{<\omega}$ in $V^{\bbP_{\omega_2}}$ is contained in a dense subset of $2^{<\omega}$ from the ground model. 
\end{clm*}
\begin{proof} Proceed by recursion on the length of the iteration. At successor stages use the argument of Lemma~\ref{single.step.density}. At limit stages, use a modification of the same argument to a generalized fusion sequence. 
\end{proof}~\end{proof}

\section{Selective Independence}

Recall the following definition:

\begin{defi} Let $\calF\subseteq [\omega]^{\leq\omega}$ be a filter. Then $\calF$ is said to be:  
	\begin{enumerate}
		\item a $p$-set if for each countable $\calF_0\subseteq\calF$ there is $C\in\calF$ such that $\forall F\in \calF_0(C\subseteq^* F)$.
		\item a $q$-set if for each finite partition $\calE$ of $\omega$ there is $C\in\calF$ such that $\forall E\in\calE(|E\cap C|\leq 1)$\footnote{A finite partition is a partition into finite sets. The set $C$ is called a semi-selector for $\calE$.}.
	\end{enumerate}
\end{defi}

\begin{defi}\label{selective.independence.def} An independent family $\calA$ is said to be selective if it is densely maximal and $\fil(\calA)$ is both a $p$-set, and a $q$-set.
\end{defi}

Selective independent families exist under CH (see~\cite{SS1, VFDM1}). Our results will show that their existence is consistent with the negation of CH.  We obtain the following preservation property of selective independent families.

\begin{thm}\label{preservation.limit} Let $\calA$ be a selective  independent family and let $\langle \bbP_\alpha,\dot{\bbQ}_\beta:\alpha\leq \delta,\beta<\delta\ra$ be a countable support iteration of proper $^\omega\omega$-bounding posets such that for each $\alpha<\delta$, $$V^{\bbP_\alpha}\vDash\calA\hbox{ is selective}.$$ Then $\calA$ remains selective in $V^{\bbP_\delta}$.
\end{thm}
\begin{proof} 
Since in $V^{\mathbb{P}_{\omega_2}}$, $\id(\calA)$ is generated by $\id(\calA)\cap V$ we have that $\fil(\calA)$ is generated by $\fil(\calA)\cap V$. Thus, $\fil(\calA)$ is a $p$-set. Moreover, since the poset $\bbP_{\omega_2}$ is $^\omega\omega$-bounding, every partition $\calE'$ of $\omega$ into finite sets in $V^{\mathbb{P}_{\omega_2}}$ is dominated by a ground model finite partition $\calE$ of $\omega$. But then a semi-selector for $\calE$ (in $V$) is a semi-selector for $\calE'$. Thus $\fil(\calA)$ remains a $q$-set (in $V^{\bbP_{\omega_2}}$). The fact that $\calA$ remains densely maximal in~$V^{\bbP_{\omega_2}}$ follows from  Lemma~\ref{dense.max.equiv}.(3) and~\cite[Lemma 3.2]{SS1} applied to $\calF=\fil(\calA)$ and $\calH=\{\omega\backslash \calA^h: h\in\FF(\calA)\}$. 
\end{proof}

\begin{thm}\label{preservation.thm} Assume CH. Let $\calA$ be a selective independent family and  $\langle \bbP_\alpha,\dot{\bbQ}_\beta:\alpha\leq\omega_2,\beta<\omega_2\ra$ be a countable support iteration of perfect trees posets. Then $\calA$ is selective in $V^{\bbP_{\omega_2}}$.
\end{thm}
\begin{proof}
It remains to show the preservation of the dense maximality of $\calA$  $V^{\bbP_{\omega_2}}$ at successor stages of the iteration. For each $\alpha\leq\omega_2$ let $V_\alpha=V^{\mathbb{P}_\alpha}$. Consider the property: 

\bigskip		
\noindent
$(\ast)_{\alpha}$ For all $h\in\FF(\calA)$ and all $X\subseteq\calA^h\cap V_\alpha$ either there is $B\in\id(\calA)\cap V$ such that  $\calA^{h'}\backslash X\subseteq B$, or there is $h'\in\FF(\calA)$ such that $h'\supseteq h$ and $\calA^{h'}\subseteq \calA^h\backslash X$. 

\bigskip
\noindent
By induction on $\alpha\leq\omega_2$, we will show that $V_\alpha\vDash (\ast)_\alpha$. This is sufficient, since one of the equivalent characterisations of dense maximality established in Lemma~\ref{dense.max.equiv} is:

\bigskip
\noindent
$(\ast)$ For all $h\in\FF(\calA)$ and all $X\subseteq\calA^h$ either there is $B\in\id(\calA)$ such that $\calA^h\backslash X\subseteq B$ or there is $h'\in\FF(\calA)$ such that $h'\supseteq h$ and $\calA^{h'}\subseteq\calA^h\backslash X$. 

\bigskip
\noindent
and moreover:

\begin{clm*}
For each $\alpha\leq\omega_2$, $V^{\mathbb{P}_\alpha}\vDash ((\ast)_{\alpha}\Leftrightarrow (\ast))$.
\end{clm*}
\begin{proof}
Fix $\alpha$ and work in $V^{\mathbb{P}_{\alpha}}$. Suppose $\neg (\ast)$. Thus, there are $h\in\FF(\calA)$ and $X\subseteq \calA^h$ such that for all $B\in\id(\calA) (=\id(\calA)^{V_\alpha})$, $\calA^h\backslash X\not\subseteq B$ and for all $h'\in\FF(\calA)$ such that $h'\supseteq h$ we have that $\calA^{h'}\not\subseteq \calA^h\backslash X$. Then clearly, for each $B\in\id(\calA)\cap V$, $\calA^h\backslash X\not\subseteq B$, since $\id(\calA)\cap V\subseteq\id(\calA)$ and for all $h'\supseteq h$ in $\FF(\calA)$, $\calA^{h'}\not\subseteq\calA^h\backslash X$. That is, $V^{\mathbb{P}_{\omega_2}}\vDash \neg(\ast)_{\omega_2}$. Now, suppose 
$V_\alpha\vDash (\ast)$. Fix $X\subseteq\calA^h$ as in $(\ast_\alpha)$. 
If $B\in\id(\calA)^{V_\alpha}$ is a witness to $(\ast)$, then any $B'\in\id(\calA)\cap V$ such that $B\subseteq B'$ is a witness to $(\ast)_\alpha$. Moreover, by Lemma~\ref{ideal.generators} there is such a $B'$.
\end{proof}

\bigskip
\noindent
Note that $(\ast)_0$ holds, since it is just saying that the independent family $\calA$ is densely maximal. Proceed by induction. Suppose we have established $(\ast)_\alpha$ and $(\ast)_{\alpha+1}$ does
not hold. Thus in $V_{\alpha+1}$ there is $h\in\FF(\calA)$ and 
$X\subseteq \calA^h$ such that for every
$B\in\id(\calA)\cap V$, $\calA^h\backslash X\not\subseteq B$ and for all $h'\in\FF(\calA)$  extending $h$ the intersection $\calA^{h'}\cap X$ is non-empty. Fix a $\bbP_{\alpha+1}$-name $\dot{X}$ for a subset of $\kappa$
and a condition $\bar{p}=(\bar{\bar{p}},\dot{p})\in\bbP_{\alpha+1}$ such that 
\begin{enumerate}
	\item $\bar{p}\Vdash \dot{X}\subseteq\calA^h$,
	\item $\bar{p}\Vdash \forall B\in\id(\calA)\cap V\hbox{ we have } (\calA^h\backslash \dot{X}\not\subseteq\check{B})$, and
	\item $\bar{p}\Vdash \forall h'\in\FF(\calA)\hbox{ s.t. }h'\supseteq h\hbox{ the set }\calA^{h'}\cap\dot{X}\hbox{ is not empty}$.
\end{enumerate}

\noindent
Let $G_\alpha$ be a $\mathbb{P}_\alpha$-generic filter such that $\bar{\bar{p}}\in G_\alpha$, $\tilde{X} =\dot{X}/G_\alpha$ and $p=\dot{p}[G_\alpha]$. Then in $V_\alpha=V[G_\alpha]$

\begin{enumerate}
	\item $p\Vdash_{\bbQ_\alpha} \tilde{X}\subseteq\calA^h$,
	\item $p\Vdash_{\bbQ_\alpha} \forall B\in(\id(\calA)\cap V\hbox{ we have } (\calA^h\backslash \tilde{X}\not\subseteq\check{B})$, and
	\item $p\Vdash_{\bbQ_\alpha} \forall h'\in\FF(\calA)\hbox{ s.t. }h'\supseteq h\hbox{ the set }\calA^{h'}\cap\tilde{X}\hbox{ is not empty}$.
\end{enumerate}

\noindent
Now, in $V_\alpha$ for each $l\in\omega$ define 
$$Y_l=\{m\in\omega: \forall t\in\hbox{split}_l(p) p_t\not\Vdash_{\bbQ_\alpha} \check{m}\notin\tilde{X}\}.$$

\begin{clm} For all $l\in\omega$, $V_\alpha\vDash (p\Vdash \tilde{X}\subseteq\check{Y}_l)\wedge (Y_l\subseteq \calA^h)$.
\end{clm}
\begin{proof}
	Fix $l\in\omega$. Take any $m\in\omega$ such that $p\Vdash\check{m}\in\tilde{X}$. Pick any $t\in\hbox{split}_l(p)$ and any $q\leq p_t$. Then $q\leq p$ and $q\Vdash \check{m}\in\tilde{X}$.  Therefore $m\in Y_l$ and so $p\Vdash\tilde{X}\subseteq \check{Y}_l$.
	
	To see that $V_\alpha\vDash Y_l\subseteq \calA^h$, consider any $m\notin\calA^h$. Since $p\Vdash\tilde{X}\subseteq\calA^h$. we must have that $p\Vdash \check{m}\notin\tilde{X}$. Then for any $t\in\hbox{split}_l(p)$, $p_t\leq p$ and $p_t\Vdash\check{m}\notin\tilde{X}$. Thus for all $q\leq p_t$, $q\Vdash\check{m}\notin\tilde{X}$ and so $m\notin Y_l$. Thus $Y_l\subseteq \calA^h$. 
\end{proof}	

Since $Y_l\in V_\alpha$ we can apply $(\ast)_\alpha$ to $\calA^h\backslash Y_l$. Thus in $V_\alpha$ either

\medskip
$(\alpha)_l$ $\exists B_l\in \id(\calA)\cap V$ such that    $\calA^h\backslash Y_l\subseteq B_l$, or

$(\beta)_l$ $\exists h'\supseteq h$ in $\FF(\calA)$ such that $\calA^{h'}\cap Y_l=\emptyset$.

\begin{clm}
	For all $l\in\omega$ property $(\alpha)_l$ holds.
\end{clm}	
\begin{proof}
	Fix $l\in\omega$ and suppose $(\beta)_l$ holds. Thus, there is $h'\in\FF(\calA)$ such that $h'\supseteq h$ and  $\calA^{h'}\cap Y_l=\emptyset$. Since  $V_\alpha\vDash (p\Vdash \tilde{X}\subseteq Y_l)$ we get $p\Vdash\tilde{X}\cap\calA^{h'}=\emptyset$, which is a contradiction to item $(3)$ in the properties of $p$. 
\end{proof}	

Thus for every $l\in\omega$ there is $B_l\in\id(\calA)\cap V$ such that $\calA^h\backslash Y_l\subseteq B_l$. Then for each $l$, the set $\calA^h\backslash Y_l\in(\id(\calA))\cap V_\alpha$ and so $\omega\backslash (\calA^h\backslash Y_l)=\omega\backslash \calA^h\cup Y_l\in\fil^{V_\alpha}(\calA)$. However $(\fil(\calA))^{V_\alpha}$ is generated by $(\fil(\calA))^V$, which is a $p$-set since $\calA$ is selective and so there is $C\in\fil(\calA)\cap V$ such that $C\subseteq^* Y_l\cup\omega\backslash\calA^h$. 

That is, for all $l\in \omega$, the set $(C\cap \calA^h)\backslash Y_l$ is bounded and so we can find $f\in{^\omega\omega}\cap V_\alpha$ such that for each  $l\in\omega$,  $(C\cap \calA^h)\backslash Y_l\subseteq f(l)$. Since $\bbP_\alpha$ is $^\omega\omega$-bounding, we can assume that $f\in V\cap{^\omega\omega}$ and that $f$ is strictly increasing. Now, for each $\gamma\in\omega$ define $\alpha_j=\min(C\cap\calA^h)\backslash (\sup_{\gamma<j} f(\gamma)+1)$. 
Thus $\{\alpha_j\}_{j\in\omega}$ is a strictly increasing sequence contained
in $C\cap\calA^h$ with the property that if $m\in C\cap \calA^h\backslash (\alpha_j+1)$ then $m\in Y_j$. Moreover, $\{\alpha_j\}_{j\in\omega}$ determines an interval partition $\calE$ of $C\cap\calA^h$ in the ground model $V$.

Since $\fil(\calA)$ is a q-set (in $V$), there is a semi-selector  
$D\in\fil(\calA)\cap V$ for $\calE$ such that $D\subseteq C$. Thus in particular, if $D=\{e_j\}_{j\in\omega}$ in increasing order, then $e_j>\alpha_j$ for each $j$. It is enough to show that there is a condition $q\leq p$ such that 
$V_\alpha\vDash \big(q\Vdash_{\bbQ_\alpha} \check{D}\subseteq \tilde{X}\big)$. Indeed, if this is the case, then
$$V_\alpha\vDash \big(q\Vdash_{\bbQ_\alpha} \calA^h\backslash\tilde{X}\subseteq \calA^h\backslash\check{D}\subseteq \omega\backslash D\big).$$
Since $D\in\fil(\calA)\cap V$, the set $\omega\backslash D\in \id(\calA)\cap V$ and since $q\leq p$ we get a contradiction to property $(2)$ of $p$. 

It remains to find $q\leq p$ such that $V_\alpha\vDash(q\Vdash_{\bbQ_\alpha} \check{D}\subseteq\tilde{X})$. 
Inductively, construct a fusion sequence $\tau=<q_j: j\in\kappa>$ below $p$ such that $q_{j}\Vdash e_j\in\tilde{X}$ for all $j\in\omega$.  Then the fusion $q$ of $\tau$ is as desired. Proceed as follows. Let $q_0=p$. Since $e_0\in Y_0$ we can find $q_0\leq p$ such that $q_0\Vdash\check{e}_0\in\tilde{X}$. Suppose we have defined
$q_i\leq_i q_{i-1}\cdots\leq_1 q_0$ such that for all $l\leq i$, $q_l\Vdash \check{e}_l\in\tilde{X}$. Consider $e_{i+1}$. Since $e_{i+1}\in Y_{i+1}$ for each $t\in\hbox{split}_{i+1}(p)$, $p_t\not\Vdash\check{e}_{i+1}\notin\tilde{X}$. Thus for all $t\in\hbox{split}_{i+1}(p)$ there is $q_t\leq p_t$ such that $q_t\Vdash \check{e}_{i+1}\in\tilde{X}$. Take $q'_{i+1}=\bigcup_{t\in\hbox{split}_{i+1}}q_t$. Then $q'_{i+1}\leq_{i+1} p$ and $q'_{i+1}\Vdash\check{e}_{i+1}\in\tilde{X}$. Find $q_{i+1}$ a common extension of $q'_{i+1}$ and $q_i$ such that $q_{i+1}\leq_{i+1} q_i$. Then $q_{i+1}$ is as desired.

Now, if $\alpha$ is a limit and $V_\beta\Vdash (\ast)_\beta$ for each $\beta<\alpha$, then by Theorem~\ref{preservation.limit} we get 
$V_\alpha\vDash (\ast)_\alpha$.
\end{proof}

\section{Compact partitions}\label{compact.partitions}

One of the most interesting open questions regarding the independence number is the consistency of $\mathfrak{i}<\mathfrak{a}$. Relying on the above preservation theorem, we obtain the consistency of $\mathfrak{i}<\mathfrak{a}_T$ where $\mathfrak{a}_T$ is the least cardinality of a maximal almost disjoint family of finitely branching subtrees of $2^{<\omega}$. Miller~\cite{AM} showed that $\mathfrak{a}_T$ is the least cardinality of a partition of $\omega^\omega$ into compact sets. Recall that:

\begin{defi} A set $A$ which is contained in $[p]$ for some perfect subtree $p$ of $2^{<\omega}$ is nowhere dense in $[p]$ if for every $s\in p$ there is $t\in p$ such that $s\subseteq t$ and 
$\{f\in 2^\omega: t\subseteq f\}\cap A=\emptyset$. 
\end{defi}

To obtain the desired consistency result, we will use the following forcing notion:

\begin{defi}(Miller~\cite{AM}) Let $\calC=\{ C_\alpha\}_{\alpha\in\omega_1}$ be a partition of $2^\omega$ into compact sets. Then $\mathbb{P}(\mathcal{C})$ is the poset of all perfect trees $p\subseteq 2^{<\omega}$ such that for each $\alpha\in\omega_1$ the set $C_\alpha\cap [p]$ is nowhere dense in $[p]$. The tree $p$ is stronger than $q$, if $p$ is a subtree of $q$. 
\end{defi}

The poset $\mathbb{P}(\mathcal{C})$ satisfies Axiom A and is selective good. It is $^\omega\omega$-bounding and has the Sacks property (see~\cite{AM, OS}). Moreover, in the $\mathbb{P}(\mathcal{C})$-generic extension, $\mathcal{C}$ is no longer a partition of $2^\omega$. 

\begin{thm} Assume CH. There is a cardinals preserving generic extension in which $$\mathfrak{i}=\omega_1<\mathfrak{a}_T=\omega_2.$$
\end{thm}
\begin{proof} Let $V$ denote the ground model and let $\mathcal{A}$ be a selective independent family in $V$. Using an appropriate bookkeeping device define a countable support iteration  $\langle \mathbb{P}_\alpha,\dot{\mathbb{Q}}_\beta:\alpha\leq\omega_2,\beta<\omega\ra$ of posets such that for each $\alpha$, $\mathbb{P}_\alpha$ forces that $\mathbb{Q}_\alpha=\mathbb{P}(\mathcal{C})$ for some uncountable partition of $2^\omega$ into compact sets and such that  $V^{\mathbb{P}_{\omega_2}}\vDash\mathfrak{a}_T=\omega_2$. By Theorem~\ref{preservation.thm}, the family $\mathcal{A}$ remains maximal independent in $V^{\mathbb{P}_{\omega_2}}$ and so a witness to $\mathfrak{i}=\omega_1$.
\end{proof}

\section{Projective wellorders}

The preservation properties discussed in the previous sections generalize to certain $S$-proper iterations, which allows us to use a specific coding technique, known as coding with perfect trees and generically adjoin a $\Delta^1_3$ definable wellorder of the reals, while preserving a co-analytic maximal independent family of cardinality $\aleph_1$. Note that both, the $\Delta^1_3$ wellorder and the $\Pi^1_1$ maximal independent families are optimal in their projective complexity.

\begin{thm}\label{selective.order} The existence of a co-analytic maximal independent family of cardinality $\aleph_1$ is consistent with the existence of a $\Delta^1_3$-definable wellorder of the reals and $\mathfrak{c}=\aleph_2$.	
\end{thm}
\begin{proof}
Work over the constructible universe $L$ and let $\mathfrak{A}=\langle (\calA_\alpha,A_\alpha):\alpha<\omega_1\ra$ be the $\Sigma^1_2$-indestructibility tower defined in~\cite[Theorem 4.11.(2)]{JBVFYK}. Then $\calA=\bigcup_{\alpha<\omega_2}$ is in fact a $\Sigma^1_2$-definable selective independent family.

Let $\bbP_{\omega_2}$ be the poset from~\cite[Section 3]{VFSF} which adjoins a $\Delta^1_3$-definable wellorder of the reals. Then in particular $\bbP_{\omega_2}=\langle \bbP_\alpha,\bbQ_\beta:\alpha\leq\omega_2,\beta<\omega_2\ra$ is a countable support iteration such that each iterand is either:
\begin{itemize}
	\item $\omega$-distributive and $S$-proper for some stationary, co-stationary subset $S$ of $\omega_1$ (these are the iterands shooting clubs through chosen stationary, co-stationary sets, see~\cite[Definition 3]{VFSF}), or
	\item proper, which does not add new reals (these are the localization iterands, see~\cite[Definition 1]{VFSF}), or
	\item  coding with perfect trees posets~\cite[Definition 2]{VFSF}
	(these are proper, $^\omega\omega$-bounding and have the Sacks property).
\end{itemize}	

Thus, in particular each iterand is $S$-proper and $^\omega\omega$-bounding.  Moreover, the proofs of Lemma~\ref{single.step.density} and Lemma~\ref{ideal.generators} generalize to coding with perfect trees and their iterations respectively.
Moreover, for each initial segment $\bbP_\alpha$ 
of the iteration, the filter $\fil(\calA)^{L^{\bbP_\alpha}}$ is generated by $\fil(\calA)\cap L$ and so remains a $p$-set. Since the iteration of $S$-proper, $^\omega\omega$-bounding posets is $S$-proper and $^\omega\omega$-bounding (see~\cite[Lemma 18]{VFSF}), $\fil(\calA)^{L^{\bbP_\alpha}}$ remains a $q$-set. The preservation of the dense maximality of $\calA$ at successor stages in which new reals are added (i.e. corresponding to coding with perfect trees) follows almost identically the proof  of Theorem~\ref{preservation.thm}, while Theorem~\ref{preservation.limit} generalizes to the iteration of $S$-proper, $^\omega\omega$-bounding posets, as the same holds for
~\cite[Lemma 3.2]{SS1}.
\end{proof}

\begin{cor}\label{definable.witnesses} The existence of a $\Delta^1_3$-wellorder of the reals is consistent with $\mathfrak{i}=\omega_1<\mathfrak{a}_T=\omega_2$. Moreover, $\mathfrak{i}$ can be witnessed by a co-analytic maximal independent family.
\end{cor}
\begin{proof}
	Using the techniques of Section~\ref{compact.partitions}, modify the proof of Theorem~\ref{selective.order} to eliminate along the iteration all compact partitions of the Baire space which are of cardinality $\aleph_1$.
\end{proof}

\end{document}